\documentclass[a4paper,12pt]{article}

\usepackage[left=2cm,right=2cm, top=2cm,bottom=3cm,bindingoffset=0cm]{geometry}

\usepackage{verbatim}
\usepackage{amsmath}
\usepackage{amsthm}
\usepackage{amssymb}
\usepackage{delarray}
\usepackage{cite}
\usepackage{hyperref}
\usepackage{mathrsfs}
\usepackage{tikz}
\usetikzlibrary{patterns}
\usepackage{caption}
\DeclareCaptionLabelSeparator{dot}{. }
\newcommand{\al}{\alpha}
\newcommand{\be}{\beta}

\newcommand{\de}{\delta}
\newcommand{\la}{\lambda}

\newcommand{\eps}{\varepsilon}
\newcommand{\vv}{\varphi}
\newcommand{\iy}{\infty}

\theoremstyle{plain}

\numberwithin{equation}{section}

\newtheorem{thm}{Theorem}[section]
\newtheorem{lem}[thm]{Lemma}

\theoremstyle{definition}

\newtheorem{alg}[thm]{Algorithm}
\newtheorem{ip}[thm]{Inverse Problem}

\theoremstyle{remark}

\newtheorem{remark}[thm]{Remark}

 \allowdisplaybreaks

\begin{document}

\begin{center}
{\Large\bf A new approach to the inverse discrete transmission eigenvalue problem}
\\[0.2cm]
{\bf Natalia P. Bondarenko and Vjacheslav A. Yurko} \\[0.2cm]
\end{center}

\vspace{0.5cm}

{\bf Abstract.} A discrete analog is considered for the inverse transmission eigenvalue problem, having applications in acoustics.
We provide a well-posed inverse problem statement, develop a constructive procedure for solving this problem, prove uniqueness of solution, global solvability, local solvability, and stability. Our approach is based on the reduction of the discrete transmission eigenvalue problem to a linear system with polynomials of the spectral parameter in the boundary condition.

\medskip


\medskip


\vspace{1cm}

\section{Introduction} \label{sec:intr}

The paper is concerned with the discrete transmission eigenvalue problem. The motivation of our study is related with the
acoustic inverse scattering problem in an inhomogeneous medium (see \cite{MP94}):
\begin{equation} \label{transPDE}
    \begin{cases}
        \Delta u + \la \rho(x) u = 0, \quad x \in B, \\
        \Delta v + \la v = 0, \quad x  \in B, \\
        u(x) = v(x), \quad \frac{\partial u(x)}{\partial n} = \frac{\partial v(x)}{\partial n}, \quad x \in \partial B,
    \end{cases}
\end{equation}
where $B$ is the ball of radius $b > 0$ in $\mathbb R^3$, $\partial B$ is its boundary, $\rho(x) > 0$ is the refraction index, $\frac{\partial}{\partial n}$ is the normal derivative. The inverse transmission eigenvalue problem consists in reconstruction of the function $\rho(x)$, which is related with the speed of sound, from the eigenvalues of the problem~\eqref{transPDE}. The majority of the studies of the inverse transmission eigenvalue problem (see \cite{MP94, MPS94, CCM07, AGP11, CL13, BYY15, BB17, BY17, GP17, XY20, BCK20}) deal with the radially symmetric case, when the problem~\eqref{transPDE} is reduced to the one-dimensional form
\begin{equation} \label{transODE}
\begin{cases}
    u'' + \la \rho(x) u = 0, \quad 0 < x < b, \\
    v'' + \la v = 0, \quad 0 < x < b, \\
    u(0) = v(0) = 0, \quad u(b) = v(b), \quad u'(b) = v'(b).
\end{cases}
\end{equation}

There were some attempts to study the inverse discrete transmission eigenvalue problems in \cite{PD11, Wei13, AP15, WW18}. In particular, the following discrete analogs of the continuous problem \eqref{transODE} have been considered:
\begin{equation} \label{dt1}
\begin{cases}
-\psi_{n + 1} + 2 \psi_n - \psi_{n-1} = \la \rho_n \psi_n, \quad n = \overline{1, l}, \\
-\psi_{n + 1}^{\bullet} + 2 \psi_n^{\bullet} - \psi_{n-1}^{\bullet} = \la \psi_n^{\bullet}, \quad n = \overline{1, l}, \\
\psi_0 = \psi_0^{\bullet} = 0, \quad \psi_l = \psi_l^{\bullet}, \quad \psi_{l + 1} = \psi_{l + 1}^{\bullet},
\end{cases}
\end{equation}
and
\begin{equation} \label{dt2}
\begin{cases}
-\psi_{n + 1} + V_n \psi_n - \psi_{n-1} = \la \psi_n, \quad n = \overline{1, l}, \\
-\psi_{n + 1}^{\bullet} + 2 \psi_n^{\bullet} - \psi_{n-1}^{\bullet} = \la \psi_n^{\bullet}, \quad n = \overline{1, l}, \\
\psi_0 = \psi_0^{\bullet} = 0, \quad \psi_l = \psi_l^{\bullet}, \quad \psi_{l + 1} = \psi_{l + 1}^{\bullet}.
\end{cases}
\end{equation}

The corresponding inverse transmission problems consist in recovering the coefficients $\{ \rho_n \}_{n = 1}^l$ or $\{ V_n \}_{n = 1}^l$ from the eigenvalues of \eqref{dt1} or \eqref{dt2}, respectively.
However, the important disadvantage of those problems is their ill-posedness, since they consist in determining $l$ unknown numbers by $(2l-1)$ or $(2l-2)$ known eigenvalues. Therefore, a small perturbation of the transmission eigenvalues influences the existence of inverse problem solution. Consequently, the known results in this direction are limited to uniqueness theorems and constructive methods of solving inverse problems, but the solvability issues remain open.

In this paper, we suggest a new approach to the inverse discrete transmission eigenvalue problem. We provide a well-posed inverse problem statement (without overdetermination in the input data), develop a constructive procedure for solving this problem, prove uniqueness of solution, global solvability, local solvability, and stability. 

Let us briefly describe our approach. Consider the generalized discrete transmission boundary value problem with respect to vectors $[\psi_n]_{n = 0}^{l + 1}$ and $[\psi_n^{\bullet}]_{n = 0}^{l + 1}$:
\begin{align} \label{eqpsi}
    & \al_n \psi_{n + 1} + \beta_n \psi_n + \al_n \psi_{n-1} = \la \psi_n, \quad n = \overline{1, l}, \\ \label{eqtpsi}
    & \al_n^{\bullet} \psi_{n + 1}^{\bullet} + \beta_n^{\bullet} \psi_n^{\bullet} + \al_n^{\bullet} \psi_{n-1}^{\bullet} = \la \psi_n^{\bullet}, \quad n = \overline{1, l}, \\ \label{bcpsi}
    & \psi_0 = \psi_0^{\bullet} = 0, \quad \psi_l = \psi_l^{\bullet}, \quad \psi_{l + 1} = \psi_{l + 1}^{\bullet},
\end{align}
where $\la$ is the spectral parameter, $\al_n, \be_n, \al_n^{\bullet}, \be_n^{\bullet} \in \mathbb C$, $n = \overline{1, l}$. 

Note that every linear system in the form
$$
a_n u_{n + 1} + b_n u_n + c_n u_{n-1} = \la \rho_n u_n, \quad n = \overline{1, l},
$$
where $a_n c_n \rho_n \ne 0$, $n = \overline{1, l}$, can be reduced to the form~\eqref{eqpsi} by the change of variables $u_n = d_n \psi_n$ with some coefficients $d_n \ne 0$, $n = \overline{1, l}$.

Assume that $\al_n \ne 0$, $\al_n^{\bullet} \ne 0$, $n = \overline{1, l}$, and $\al_l \ne \al_l^{\bullet}$. Then, the boundary value problem \eqref{eqpsi}-\eqref{bcpsi} has $(2l-1)$  eigenvalues $\{ \la_j \}_{j = 1}^{2l-1}$ (counting with multiplicities). Our inverse discrete transmission problem is stated as follows.

\begin{ip} \label{ip:psi}
Suppose that $\{ \al_n^{\bullet} \}_{n = 1}^l$, $\{ \be_n^{\bullet} \}_{n = 1}^l$, and $\al_l$ are known a priori. Given the eigenvalues $\{ \la_j \}_{j = 1}^{2l-1}$, find $\{ \al_n \}_{n = 1}^{l-1}$ and $\{ \be_n \}_{n = 1}^l$.
\end{ip}

Our method of solution is based on the reduction of the problem \eqref{eqpsi}-\eqref{bcpsi} to the form
\begin{align} \label{eqy}
    & a_n y_{n + 1} + b_n y_n + y_{n-1} = \la y_n, \quad n = \overline{1, l}, \\ \label{bcR}
    & R_0(\la) y_1 - R_1(\la) y_0 = 0, \quad y_{l + 1} = 0,
\end{align}
where $a_n, b_n \in \mathbb C$, $a_n \ne 0$, $n = \overline{1, l}$, and $R_0(\la)$, $R_1(\la)$ are relatively prime polynomials (i.e. not having common roots). The latter polynomials are constructed by the coefficients $\{ \al_n^{\bullet} \}_{n = 1}^l$, $\{ \be_n^{\bullet} \}_{n = 1}^l$, $\al_l$ so that the eigenvalues of the problem \eqref{eqy}-\eqref{bcR} coincide with the eigenvalues of \eqref{eqpsi}-\eqref{bcpsi}. Thus, Inverse Problem~\ref{ip:psi} is reduced to the following problem.

\begin{ip} \label{ip:R}
    Given eigenvalues $\{ \la_j \}_{j = 1}^{2l-1}$ of the problem \eqref{eqy}-\eqref{bcR}, find $\{ a_n \}_{n = 1}^{l-1}$ and $\{ b_n \}_{n = 1}^l$.
\end{ip}

Note that the coefficient $a_l$ is multiplied by $y_{l + 1} = 0$ in \eqref{eqy}, so this coefficient cannot be recovered. Although we are primarily interested in Inverse Problem~\ref{ip:R} in connection with Inverse Problem~\ref{ip:psi}, our method for solving Inverse Problem~\ref{ip:R} is developed for the case of arbitrary relatively prime polynomials $R_0(\la)$ and $R_1(\la)$ satisfying some additional conditions. We reduce Inverse Problem~\ref{ip:R} to the reconstruction of $\{ a_n \}_{n = 1}^{l-1}$ and $\{ b_n \}_{n = 1}^l$ from the Weyl coefficients defined in Section~\ref{sec:aux}. The latter problem is equivalent to the classical inverse problem that consists in the reconstruction of the coefficients from the two spectra $\{ \mu_j \}_{j = 1}^l$ and $\{ \nu_j \}_{j = 1}^{l-1}$ of the eigenvalue problems for equations~\eqref{eqy} with the boundary conditions $y_0 = y_{l + 1} = 0$ and $y_1 = y_{l + 1} = 0$, respectively. In order to deal with this classical inverse problem, we adapt the methods of Yurko \cite{Yur96, Yur00}. Note that in \cite{Yur96, Yur00} these methods are developed for arbitrary discrete systems of triangular structure corresponding to arbitrary-order differential operators. 

It is worth mentioning that \eqref{eqy}-\eqref{bcR} is the discrete analog of the Sturm-Liouville problem
\begin{equation} \label{entire}
    \begin{cases}
        -y'' + q(x) y = \la y, \quad x \in (0, \pi), \\
        y(0) = 0, \quad f_1(\la) y'(\pi) + f_2(\la) y(\pi) = 0,
    \end{cases}
\end{equation}
with entire analytic functions $f_1(\la)$ and $f_2(\la)$ in the boundary condition. Bondarenko \cite{Bond20-Open, Bond20-mmas, Bond21} has developed a unified approach, which is based on the reduction to the form~\eqref{entire}, to a wide class of inverse problems for differential operators, including the Hochstadt-Lieberman half-inverse problem~\cite{HL78}, the inverse transmission eigenvalue problem \cite{MP94, BB17}, partial inverse problems on metric graphs \cite{Bond18, Bond21}. The approach of \cite{Bond20-Open, Bond20-mmas, Bond21} allows not only to prove the uniqueness theorems, but also to develop constructive algorithms for solution and to obtain solvability conditions for various types of partial inverse problems. In this paper, the idea of such reduction is transferred to the discrete case.

The paper is organized as follows. In Section~\ref{sec:aux}, the solution of the auxiliary inverse problem for equations~\eqref{eqy} by the Weyl coefficients is described. In Section~\ref{sec:main}, we provide our main results concerning Inverse Problems~\ref{ip:psi}-\ref{ip:R}. In Section~\ref{sec:Hochstadt}, we discuss the connection of our results with the classical results of Hochstadt~\cite{Hoch74, Hoch79}. In particular, we show that the inverse problem for a Jacobi matrix by mixed data \cite{Hoch79} can be easily reduced to Inverse Problem~\ref{ip:R} and solved by our method.

\section{Auxiliary inverse problem} \label{sec:aux}

Denote by $[P_n]_{n = 0}^{l + 1}$, $[Q_n]_{n = 0}^{l + 1}$, and $[\Phi_n]_{n = 0}^{l + 1}$ the solutions of the system~\eqref{eqy} satisfying the initial conditions
$$
P_0(\la) = 0, \quad P_1(\la) = 1, \quad Q_0(\la) = 1, \quad Q_1(\la) = 0,
$$
and the boundary conditions 
\begin{equation} \label{bcPhi}
\Phi_0(\la) = 1, \quad \Phi_{l + 1}(\la) = 0,
\end{equation}
respectively. Obviously,
\begin{equation} \label{PhiPQ}
\Phi_n(\la) = Q_n(\la) + M(\la) P_n(\la), \quad n = \overline{1, l},
\end{equation}
where $M(\la)$ is called \textit{the Weyl function} of~\eqref{eqy}.
Using \eqref{bcPhi}, we obtain
\begin{equation} \label{fracM}
M(\la) = -\frac{Q_{l + 1}(\la)}{P_{l + 1}(\la)}.
\end{equation}

Below the notation $f(\la) \sim f_0 \la^k$ means that $f(\la)$ is a polynomial of degree $k$ with the leading coefficient $f_0$. By induction, we show that
\begin{equation} \label{asymptP}
P_{n + 1}(\la) \sim \left( \prod_{k = 1}^n a_k^{-1} \right) \la^n, \quad
Q_{n + 1}(\la) \sim -\left( \prod_{k = 1}^n a_k^{-1} \right) \la^{n-1}, \quad n = \overline{1, l}.
\end{equation}
Therefore, $M(\la)$ is a rational function, so it can be represented in the form
\begin{equation} \label{seriesM}
M(\la) = \sum_{k = 1}^{\iy} M_k \la^{-k}, \quad M_1 = 1,
\end{equation}
where the series uniformly converges for sufficiently large $r > 0$, $|\la| \ge r$. We call the numbers $\{ M_k \}_{k = 1}^{2l}$ \textit{the Weyl coefficients}.

\begin{lem} \label{lem:Weyl}
The Weyl coefficients $\{ M_k \}_{k = 1}^{2l}$ uniquely specify $M(\la)$.
\end{lem}

\begin{proof}
Suppose that the Weyl functions $M(\la)$ and $\tilde M(\la)$ have the Weyl coefficients $\{ M_k \}_{k = 1}^{2l}$ and $\{ \tilde M_k \}_{k = 1}^{2l}$, respectively. Using \eqref{fracM}, we obtain
$$
\frac{Q_{l + 1}(\la)}{P_{l + 1}(\la)} - \frac{\tilde Q_{l + 1}(\la)}{\tilde P_{l + 1}(\la)} = O\left(\la^{-(2l+1)}\right), \quad |\la| \to \iy.
$$
This implies
$$
Q_{l + 1}(\la) \tilde P_{l + 1}(\la) - P_{l + 1}(\la) \tilde Q_{l + 1}(\la) \equiv 0.
$$
Hence $M(\la) \equiv \tilde M(\la)$.
\end{proof}

\begin{ip} \label{ip:Weyl}
Given the Weyl coefficients $\{ M_k \}_{k = 2}^{2l}$, find $\{ a_n \}_{n = 1}^{l-1}$ and $\{ b_n \}_{n = 1}^l$.
\end{ip}

The coefficient $a_l$ cannot be recovered because of the condition $\Phi_{l + 1} = 0$. Without loss of generality assume that $a_l = 1$.

Denote by $\{ \mu_j \}_{j = 1}^l$ and $\{ \nu_j \}_{j = 1}^{l-1}$ the zeros of the polynomials $P_{l + 1}(\la)$ and $Q_{l + 1}(\la)$, respectively (counting with multiplicities). Clearly, $\{ \mu_j \}_{j = 1}^l$ and $\{ \nu_j \}_{j = 1}^{l-1}$ are the eigenvalues of the boundary value problems for the system~\eqref{eqy} with the boundary conditions $y_0 = y_{l + 1} = 0$ and $y_1 = y_{l + 1} = 0$, respectively. For $\{ \nu_j \}_{j = 1}^{l-1}$, it is supposed that $n = \overline{2, l}$ in \eqref{eqy}.
In view of \eqref{fracM}, \eqref{asymptP}, and Lemma~\ref{lem:Weyl}, the two spectra $\{ \mu_j \}_{j = 1}^l$ and $\{ \nu_j \}_{j = 1}^{l-1}$ uniquely specify $M(\la)$, and vice versa. Hence, Inverse Problem~\ref{ip:Weyl} is equivalent to the following classical inverse problem by two spectra.

\begin{ip} \label{ip:2sp}
Given $\{ \mu_j \}_{j = 1}^l$ and $\{ \nu_j \}_{j = 1}^{l-1}$, find $\{ a_n \}_{n = 1}^{l-1}$ and $\{ b_n \}_{n = 1}^l$.
\end{ip}

Proceed with constructive solution of Inverse Problem~\ref{ip:Weyl}.
In view of \eqref{asymptP}, we have
\begin{equation} \label{polyP}
P_{n + 1}(\la) = \sum_{i = 0}^n c_{in} \la^i, \quad n = \overline{0, l},
\end{equation}
where $c_{in} \in \mathbb C$, $c_{nn} \ne 0$. Substituting \eqref{polyP} into \eqref{eqy}, we obtain
$$
a_n c_{nn} = c_{n-1,n-1}, \quad a_n c_{n-1,n} + b_n c_{n-1,n-1} = c_{n-2, n-1}, \quad n = \overline{1, l},
$$
where $c_{-1,0} = 0$. Hence
\begin{equation} \label{calcab}
a_n = \frac{c_{n-1,n-1}}{c_{nn}}, \quad b_n = \frac{c_{n-2,n-1} - a_n c_{n-1,n}}{c_{n-1,n-1}}, \quad n = \overline{1, l}.
\end{equation}

Denote by $[v_n]_{n = 0}^{l+1}$ the solution of equation~\eqref{eqy} satisfying the initial conditions $v_{l + 1} = 0$, $v_l = 1$. One can easily show that
\begin{equation} \label{Phiv}
\Phi_n(\la) = \frac{v_n(\la)}{v_0(\la)}, \quad n = \overline{0, l + 1},
\end{equation}
and $v_n(\la) \sim \la^{l-n}$. Consequently,
\begin{equation} \label{asymptPhi}
\Phi_n(\la) = \la^{-n} + O\left( \la^{-(n + 1)}\right), \quad |\la| \to \iy,
\end{equation}
that is, there are the zero coefficients at $\la^{-k}$, $k = \overline{1, n-1}$. Using \eqref{PhiPQ}, \eqref{seriesM}, \eqref{polyP}, and \eqref{asymptPhi}, we obtain the relations
\begin{equation} \label{relcM}
\sum_{i = 0}^n c_{in} M_{i + k + 1} = \de_{nk}, \quad k = \overline{0, n}, \quad n = \overline{1, l},
\end{equation}
where $\de_{nk}$ is the Kronecker delta. 

\begin{thm} \label{thm:nscWeyl}
For complex numbers $\{ M_k \}_{k = 1}^{2l}$ $(M_1 = 1)$ to be the Weyl coefficients of \eqref{eqy} with some $\{ a_n \}_{n = 1}^l$ and $\{ b_n \}_{n = 1}^l$, $a_n \ne 0$, $n = \overline{1, l-1}$, $a_l = 1$, the following condition is necessary and sufficient: 
\begin{equation} \label{Dcond}
\Delta_n := \det([M_{i + j - 1}]_{i,j = 1}^{n + 1}) \ne 0, \quad n = \overline{1, l-1}.
\end{equation}
Under the condition~\eqref{Dcond}, the solution of Inverse Problem~\ref{ip:Weyl} is unique.
\end{thm}

\begin{proof}
\textit{Necessity.} Let $\{ M_k \}_{k = 1}^{2l}$ be the Weyl coefficients corresponding to some $\{ a_n \}_{n = 1}^l$ and $\{ b_n \}_{n = 1}^l$. By the above construction, there exist $c_{in}$, $i = \overline{0, n}$, $n = \overline{1, l}$, satisfying \eqref{relcM} and $c_{nn} \ne 0$. Obviously, $\Delta_0 = 1$. Suppose that $\Delta_n \ne 0$ and $\Delta_{n + 1} = 0$ for some $n \in \{ 0, \ldots, l-2 \}$. Then the linear system~\eqref{relcM} is inconsistent for this value of $n$. This contradiction yields the claim.

\textit{Sufficiency.} Suppose that $\Delta_n \ne 0$, $n = \overline{1, l-1}$. Then, for each fixed $n = \overline{1, l-1}$, the linear system \eqref{relcM} is uniquely solvable with respect to $\{ c_{in} \}_{i = 0}^n$. In particular, 
$$
c_{nn} = \frac{\Delta_{n-1}}{\Delta_n} \ne 0, \quad n = \overline{1, l-1}.
$$
Put $c_{00} := 1$, $c_{ll} := c_{l-1,l-1}$, $\Delta_l := \Delta_{l-1} c_{ll}^{-1}$. Find $\{ c_{il} \}_{i = 0}^{l-1}$ from the system~\eqref{relcM} for $n = l$ by Cramer's rule, replacing the determinant of this system by $\Delta_l$. (We cannot directly calculate $\det([M_{i + j - 1}]_{i,j = 1}^{l + 1})$, since $M_{2l+1}$ is unknown). Using \eqref{calcab}, find $\{ a_n \}_{n = 1}^l$ and $\{ b_n \}_{n = 1}^l$. Clearly, $a_n \ne 0$, $n = \overline{1, l-1}$, and $a_l = 1$. It remains to prove that $\{ M_k \}_{k = 1}^{2l}$ are the Weyl coefficients corresponding to $\{ a_n \}_{n = 1}^l$ and $\{ b_n \}_{n = 1}^l$. Suppose that equation~\eqref{eqy} has the Weyl coefficients $\{ \tilde M_k \}_{k = 1}^{2l}$, $\tilde M_l = 1$. Then the relations \eqref{relcM} are valid for $M_n$ replaced by $\tilde M_n$. Using this fact together with $c_{nn} \ne 0$, $n = \overline{1, l}$, one can easily show that $M_k = \tilde M_k$, $k = \overline{1, 2l}$. The described algorithm of unique reconstruction of $\{ a_n \}_{n = 1}^{l-1}$ and $\{ b_n \}_{n = 1}^l$ by $\{ M_k \}_{k = 1}^{2l}$ implies the uniqueness of the inverse problem solution.
\end{proof}

Thus, we have arrived at the following constructive algorithm for solving Inverse Problem~\ref{ip:Weyl}.

\begin{alg} \label{alg:Weyl}
Let $\{ M_k \}_{k = 1}^{2l}$, $M_1 = 1$ be given. We have to find $\{ a_n \}_{n = 1}^{l-1}$ and $\{ b_n \}_{n = 1}^l$.

\begin{enumerate}
    \item For each $n = \overline{1, l-1}$, find $\{ c_{in} \}_{i = 0}^n$ by solving the system~\eqref{relcM}.
    \item Put $c_{00} := 1$, $c_{ll} := c_{l-1,l-1}$, $\Delta_l := \Delta_{l-1} c_{ll}^{-1}$, and find $\{ c_{il} \}_{i= 0}^{l-1}$ by solving the system~\eqref{relcM} for $n = l$, replacing its determinant by $\Delta_l$.
    \item Find $\{ a_n \}_{n = 1}^{l-1}$ and $\{ b_n \}_{n = 1}^l$ by~\eqref{calcab}.
\end{enumerate}
\end{alg}

Using Theorem~\ref{thm:nscWeyl} and Algorithm~\ref{alg:Weyl}, we easily obtain the following theorem on local solvability and stability of Inverse Problem~\ref{ip:Weyl}.

\begin{thm} \label{thm:locWeyl}
Let $\{ M_k \}_{k = 1}^{2l}$ be the Weyl coefficients corresponding to $\{ a_n \}_{n = 1}^l$ and $\{ b_n \}_{n = 1}^l$, $a_n \ne 0$, $n = \overline{1, l-1}$, $a_l = 1$. Then there exists $\eps > 0$ (depending on $\{ a_n \}_{n = 1}^l$ and $\{ b_n \}_{n = 1}^l$) such that, for any complex numbers $\{ \tilde M_k \}_{k = 2}^{2l}$ satisfying the estimate
$$
\de := \max\limits_{k = \overline{2, 2l}} |M_k - \tilde M_k| \le \eps,
$$
there exist unique complex numbers $\{ \tilde a_n \}_{n = 1}^l$ and $\{ \tilde b_n \}_{n = 1}^l$ such that $\tilde a_n \ne 0$, $n = \overline{1,l-1}$, $\tilde a_l = 1$, and $\{ \tilde M_k \}_{k = 2}^{2l}$ are the corresponding Weyl coefficients. Moreover, the following estimates hold:
$$
    |a_n - \tilde a_n| \le C \de, \quad |b_n - \tilde b_n| \le C \de, \quad n = \overline{1, l},
$$
where the constant $C$ depends only on $\{ a_n \}_{n = 1}^l$ and $\{ b_n \}_{n = 1}^l$.
\end{thm}

\section{Main results} \label{sec:main}

In this section, we investigate Inverse Problems~\ref{ip:psi}-\ref{ip:R}. The uniqueness of solution, global solvability, local solvability, stability are proved, and also reconstruction algorithms are obtained for these two inverse problems.

Denote by $[\vv_n]_{n = 0}^{l + 1}$ and $[\vv_n^{\bullet}]_{n = 0}^{l + 1}$ the solutions of equations~\eqref{eqpsi} and~\eqref{eqtpsi}, respectively, satisfying the initial conditions $\vv_0(\la) = \vv_0^{\bullet}(\la) = 0$, $\vv_1(\la) = \vv_1^{\bullet}(\la) = 1$. Clearly, the eigenvalues of the boundary value problem~\eqref{eqpsi}-\eqref{bcpsi} coincide with the zeros of the characteristic function
$$
D(\la) = \vv_l(\la) \vv_{l + 1}^{\bullet}(\la) - \vv_{l + 1} \vv_l^{\bullet}(\la).
$$
By induction, we show that
\begin{equation} \label{asymptphi}
\vv_n(\la) \sim \left( \prod_{k = 1}^{n-1} \frac{1}{\al_k} \right) \la^{n-1}, \quad
\vv_n^{\bullet}(\la) \sim \left( \prod_{k = 1}^{n-1} \frac{1}{\al_k^{\bullet}} \right) \la^{n-1}, \quad n = \overline{1, l + 1}.
\end{equation}
Consequently,
$$
D(\la) \sim \left( \prod_{k = 1}^{l} \frac{1}{\al_k \al_k^{\bullet}}\right) (\al_l^{\bullet} - \al_l) \la^{2l-1}.
$$
Recall that $\al_k \ne 0$, $\al_k^{\bullet} \ne 0$, $k = \overline{1, l}$, and $\al_l \ne \al_l^{\bullet}$. Therefore, the problem~\eqref{eqpsi}-\eqref{bcpsi} has exactly $(2l-1)$ eigenvalues $\{ \la_j \}_{j = 1}^{2l-1}$ (counting with multiplicities).

Denote
\begin{gather} \nonumber
    d_{l + 1} := 1, \quad d_n := \al_n^{-1} d_{n + 1}, \quad n = l, l-1, \ldots, 1, \\ \nonumber
    a_{l + 1 - n} := d_n \al_n d_{n-1}^{-1}, \quad n = \overline{2,l}, \quad a_l := 1,  \\ \nonumber
    b_{l + 1 - n} := \beta_n, \quad
    y_{l + 1- n} = d_n \psi_n, \quad n = \overline{1, l}, \\ \label{defR}
    R_0(\la) := \al_l \vv_{l + 1}^{\bullet}(\la), \quad R_1(\la) = \vv_l^{\bullet}(\la).
\end{gather}
This change of variables reduces the eigenvalue problem~\eqref{eqpsi}-\eqref{bcpsi} to the equivalent form~\eqref{eqy}-\eqref{bcR}. On the other hand, if the coefficients $\{ a_n \}_{n = 1}^l$, $\{ b_n \}_{n = 1}^l$, and $\al_l$ are known, one can find $\{ \al_n \}_{n = 1}^{l-1}$ and $\{ \beta_n\}_{n = 1}^l$ by using the formulas
\begin{gather} \label{calcd}
    d_{l + 1} := 1, \quad d_l := \al_l^{-1}, \quad
    d_{n - 1} := d_{n + 1} a_{l + 1 - n}^{-1}, \quad n = l, l-1, \ldots, 2, \\ \label{calcalbe}
    \al_n := d_n^{-1} d_{n + 1}, \quad \beta_n := b_{l + 1 - n}, \quad n = \overline{1, l}.
\end{gather}

It follows from \eqref{asymptphi}-\eqref{defR} and the condition $\al_l \ne \al_l^{\bullet}$ that
\begin{gather} \label{asymptR}
    R_0(\la) \sim r_0 \la^l, \quad R_1(\la) \sim r_1 \la^{l-1}, \quad r_0 \ne r_1, \quad r_j \ne 0, \: j = 0, 1, \\ \label{r01}
    r_0 = \frac{\al_l}{\al_l^{\bullet}} r_1, \quad r_1 = \prod_{k = 1}^{l-1} \frac{1}{\al_k^{\bullet}}.
\end{gather}

Let us consider the problem~\eqref{eqy}-\eqref{bcR} in the general from with arbitrary relatively prime polynomials $R_0(\la)$ and $R_1(\la)$ satisfying~\eqref{asymptR}. The eigenvalues of \eqref{eqy}-\eqref{bcR} coincide with the zeros of the characteristic function
\begin{equation} \label{defE}
E(\la) = R_0(\la) v_1(\la) - R_1(\la) v_0(\la).
\end{equation}
Since $v_n \sim \la^{l-n}$, we have
\begin{equation} \label{asymptE}
E(\la) \sim (r_0 - r_1) \la^{2l-1}.
\end{equation}
Hence, under our assumptions, the problem \eqref{eqy}-\eqref{bcR} has exactly $(2l-1)$ eigenvalues $\{ \la_j \}_{j = 1}^{2l-1}$ (counting with multiplicities). If \eqref{defR} holds, then $\{ \la_j \}_{j = 1}^{2l-1}$ coincide with the eigenvalues of the problem~\eqref{eqpsi}-\eqref{bcpsi}. Clearly, the polynomial $E(\la)$ can be constructed by its zeros as follows:
\begin{equation} \label{prodE}
E(\la) = (r_0 - r_1) \prod_{j = 1}^{2l-1} (\la - \la_j).
\end{equation}

\begin{lem} \label{lem:uniqv}
Suppose that $R_0(\la)$, $R_1(\la)$, $E(\la)$ are arbitrary polynomials satisfying~\eqref{asymptR}, \eqref{asymptE}, and $R_0(\la)$, $R_1(\la)$ are relatively prime. Then, there exist unique polynomials $v_j(\la) \sim \la^{l-j}$, $j = 0, 1$, satisfying~\eqref{defE}.
\end{lem}

\begin{proof}
Denote by $\{ \xi_{n,j} \}_{n = 1}^{l-j}$ the zeros of the polynomial $R_j(\la)$, $j = 0, 1$ (counting with multiplicities). Introduce the notations
\begin{gather*}
S_j := \{ 1 \} \cup \{ n = \overline{2,l-j} \colon \xi_{n,j} \ne \xi_{n-1,j} \}, \\
m_{n,j} := \# \{ k = \overline{1, l-j} \colon \xi_{k,j} = \xi_{n,j} \}, \quad n \in S_j, \quad j = 0, 1,
\end{gather*}
that is, $S_j$ is the set of indices of all the distinct values among $\{ \xi_{n,j} \}_{n= 1}^{l-j}$ and $m_{n,j}$ is the multiplicity of $\xi_{n,j}$. Denote
$$
f^{<\nu>}(\la) = \frac{1}{\nu!} \frac{d^{\nu} f(\la)}{d\la^{\nu}}, \quad \nu \ge 0.
$$

Using \eqref{defE}, we obtain
$$
E^{<\nu>}(\xi_{n,j}) = (-1)^{1-j} \sum_{s = 0}^{\nu} R_{1-j}^{<\nu - s>}(\xi_{n,j}) v_j^{<s>}(\xi_{n,j}), \quad n \in S_j, \quad \nu = \overline{0, m_{n,j}-1}, \quad j = 0, 1.
$$
Hence
\begin{gather} \label{findv}
v_j^{<\nu>}(\xi_{n,j}) = R_{1-j}^{-1}(\xi_{n,j}) \left( (-1)^{1-j} E^{<\nu>}(\xi_{n,j}) - \sum_{s = 0}^{\nu - 1} R_{1-j}^{<\nu - s>}(\xi_{n,j}) v_j^{<s>}(\xi_{n,j})\right), \\ \nonumber
 n \in S_j, \quad \nu = \overline{0, m_{n,j}-1}, \quad j = 0, 1.
\end{gather}
Note that $R_{1-j}(\xi_{n,j}) \ne 0$, since $R_0(\la)$ and $R_1(\la)$ do not have common roots. Clearly, the polynomials $v_j(\la) \sim \la^{l-j}$ can be uniquely constructed by the values $\{ v_j^{<\nu>}(\xi_{n,j}) \}_{n \in S_j, \, \nu = \overline{0, m_{n,j}-1}}$, $j = 0, 1$, by using the Hermite interpolation.
\end{proof}

Although the proof of Lemma~\ref{lem:uniqv} is constructive, it is more convenient to use another method for finding $v_0(\la)$ and $v_1(\la)$. Represent the polynomials appearing in \eqref{defE} in the form
\begin{equation} \label{poly}
    E(\la) = \sum_{s = 0}^{2l-1} e_s \la^s, \quad
    R_j(\la) = \sum_{n = 0}^{l-j} r_{n,j} \la^n, \quad
    v_j(\la) = \sum_{k = 0}^{l-j} v_{k,j} \la^k, \quad j = 0, 1.
\end{equation}
Note that
\begin{equation} \label{lead}
    e_{2l-1} = r_0 - r_1, \quad r_{l-j, j} = r_j, \quad v_{l-j,j} = 1, \quad j = 0, 1.
\end{equation}
Substituting \eqref{poly} into~\eqref{defE} and taking~\eqref{lead} into account, we obtain the following system of $(2l-1)$ linear equations
\begin{equation} \label{sysv}
    \sum_{n + k = s} r_{n,0} v_{k,1} - \sum_{n + k = s} r_{n,1} v_{k,0} = e_s, \quad s = \overline{0, 2l-2},
\end{equation}
with respect to the $(2l-1)$ unknown values $\{ v_{k, j} \}_{k = 0}^{l-j-1}$, $j = 0, 1$. By virtue of Lemma~\ref{lem:uniqv}, the system~\eqref{sysv} is uniquely solvable, so its determinant is non-zero. Solving \eqref{sysv}, one can find $\{ v_{k, j} \}_{k = 0}^{l-j-1}$, $j = 0, 1$ and construct the polynomials $v_0(\la)$, $v_1(\la)$.

It follows from \eqref{PhiPQ} and \eqref{Phiv} that
$M(\la) = \dfrac{v_1(\la)}{v_0(\la)}$.
Therefore, using $v_j(\la)$, $j = 0, 1$, one can construct the Weyl function $M(\la)$ and the Weyl coefficients $\{ M_k \}_{k = 1}^{2l}$, and then solve Inverse Problem~\ref{ip:Weyl} to determine $\{ a_n \}_{n = 1}^{l-1}$ and $\{ b_n \}_{n = 1}^l$. Thus, we have arrived at the following algorithm for solving Inverse Problem~\ref{ip:R}.

\begin{alg} \label{alg:R}
Suppose that the relatively prime polynomials $R_j(\la)$, $j = 0, 1$, satisfying~\eqref{asymptR} and the complex numbers $\{ \la_j \}_{j = 1}^{2l-1}$ are given. We have to find $\{ a_n \}_{n = 1}^{l-1}$ and $\{ b_n \}_{n = 1}^l$.

\begin{enumerate}
    \item Construct $E(\la)$ by~\eqref{prodE}.
    \item Put $v_{l-j,j} = 1$, $j = 0, 1$, find $\{ v_{k, j} \}_{k = 0}^{l-j-1}$, $j = 0, 1$, by solving the linear system~\eqref{sysv}, and construct the polynomials $v_0(\la)$, $v_1(\la)$ by \eqref{poly}.
    \item Construct the coefficients $\{ M_k \}_{k = 1}^{2l}$ in the expansion~\eqref{seriesM} of the function $M(\la) = \dfrac{v_1(\la)}{v_0(\la)}$.
    \item Apply Algorithm~\ref{alg:Weyl} to find $\{ a_n \}_{n = 1}^{l-1}$ and $\{ b_n \}_{n = 1}^l$ by $\{ M_k \}_{k = 1}^{2l}$.
\end{enumerate}
\end{alg}

\begin{remark}
For definiteness, we assume throughout the paper that $\al_l \ne \al_l^{\bullet}$, $r_0 \ne r_1$. However, our results are also valid for the opposite case with some technical modifications.
Indeed, if $\al_l = \al_l^{\bullet}$ and $r_0 = r_1$, then $\deg(E) < 2l-1$, and the number $p$ of the eigenvalues $\{ \la_j \}_{j = 1}^p$ is less than $(2l-1)$. In this case, the polynomial $E(\la)$ has the form
\begin{equation} \label{prodEc}
    E(\la) = c \prod_{j = 1}^p (\la - \la_j)
\end{equation}
instead of \eqref{prodE}. The constant $c \ne 0$ should be added to the initial data of the inverse problem. Thus, given the eigenvalues $\{ \la_j \}_{j = 1}^p$ and the constant $c$, one can use \eqref{prodEc} to construct $E(\la)$. It is easy to check that Lemma~\ref{lem:uniqv}, Algorithm~\ref{alg:R}, and the subsequent results remain valid with necessary technical changes.  
\end{remark}

Using Theorem~\ref{thm:nscWeyl}, we immediately obtain necessary and sufficient conditions of solvability for Inverse Problem~\ref{ip:R} together with the uniqueness of solution.

\begin{thm} \label{thm:nscR}
Suppose that $R_0(\la)$ and $R_1(\la)$ are relatively prime polynomials satisfying~\eqref{asymptR}. Then, for complex numbers $\{ \la_j \}_{j = 1}^{2l-1}$ to be the eigenvalues of a problem \eqref{eqy}-\eqref{bcR}, it is necessary and sufficient that the numbers $\{ M_k \}_{k = 1}^{2l-1}$ constructed by steps 1-3 of Algorithm~\ref{alg:R} fulfill the condition \eqref{Dcond}. Under the latter condition, the solution of Inverse Problem~\ref{ip:R} is unique.
\end{thm}

Using the reduction of the problem~\eqref{eqpsi}-\eqref{bcpsi} to \eqref{eqy}-\eqref{bcR}, we obtain the algorithm for solving the discrete inverse transmission eigenvalue problem (Inverse Problem~\ref{ip:psi}).

\begin{alg} \label{alg:psi}
Suppose that $\{ \al_n^{\bullet} \}_{n = 1}^l$, $\{ \beta_n^{\bullet}\}_{n = 1}^l$, $\al_l$, and $\{ \la_j \}_{j = 1}^{2l-1}$ are given. We have to find $\{ \al_n \}_{n = 1}^{l-1}$ and $\{ \beta_n \}_{n = 1}^l$.

\begin{enumerate}
    \item Find the solution $[\vv_n^{\bullet}]_{n = 0}^{l + 1}$ of equation~\eqref{eqtpsi} satisfying the initial conditions $\vv_0^{\bullet} = 0$, $\vv_1^{\bullet} = 1$.
    \item Construct the numbers $r_0$, $r_1$ by \eqref{r01} and the polynomials $R_0(\la)$, $R_1(\la)$ by \eqref{defR}.
    \item Using $R_j(\la)$, $j = 0, 1$, and $\{ \la_j \}_{j = 1}^{2l-1}$, construct $\{ a_n \}_{n = 1}^{l-1}$ and $\{ b_n \}_{n = 1}^l$ by Algorithm~\ref{alg:R}.
    \item Find $\{ \al_n \}_{n = 1}^{l-1}$ and $\{ \beta_n \}_{n = 1}^l$ by \eqref{calcd}-\eqref{calcalbe}.
\end{enumerate}
\end{alg}

Relying on Algorithm~\ref{alg:psi} and Theorems~\ref{thm:nscWeyl} and \ref{thm:nscR}, we obtain necessary and sufficient conditions of solvability and the uniqueness of solution for Inverse Problem~\ref{ip:psi}.

\begin{thm} \label{thm:nscpsi}
Suppose that $\{ \al_n^{\bullet} \}_{n = 1}^l$, $\{ \beta_n^{\bullet} \}_{n = 1}^l$, and $\al_l$ are fixed complex numbers, $\al_n^{\bullet} \ne 0$, $n = \overline{1, l}$, $\al_l \ne 0$, $\al_l \ne \al_l^{\bullet}$. Then, for complex numbers $\{ \la_j \}_{j = 1}^{2l-1}$ to be the eigenvalues of a problem \eqref{eqpsi}-\eqref{bcpsi}, it is necessary and sufficient that the numbers $\{ M_k \}_{k = 1}^{2l-1}$ constructed by steps~1-3 of Algorithm~\ref{alg:psi} and by steps 1-3 of Algorithm~\ref{alg:R} fulfill the condition \eqref{Dcond}. Under the latter condition, the solution of Inverse Problem~\ref{ip:psi} is unique.
\end{thm}

\begin{remark} \label{rem:imp}
Theorems~\ref{thm:nscR} and~\ref{thm:nscpsi} contain the implicit condition~\eqref{Dcond}. Observe that analogous implicit conditions arise in investigation of the half-inverse problem for the Sturm-Liouville differential equation (see \cite{HM04, MP10}). It seems that such kind of conditions is unavoidable for these problems.
\end{remark}

Since each step of Algorithms~\ref{alg:R} and~\ref{alg:psi} is stable, we obtain the following theorems on local solvability and stability of Inverse Problems~\ref{ip:psi}-\ref{ip:R}.

\begin{thm} \label{thm:locR}
Let $R_0(\la)$ and $R_1(\la)$ be fixed relatively prime polynomials satisfying~\eqref{asymptR}, let $\{ a_n \}_{n = 1}^l$, $\{ b_n \}_{n = 1}^l$ be fixed complex numbers, $a_n \ne 0$, $n = \overline{1, l-1}$, $a_l = 1$, and let $\{ \la_j \}_{j = 1}^{2l-1}$ be the eigenvalues of the corresponding problem~\eqref{eqy}-\eqref{bcR}. Then, there exists $\eps > 0$ (depending on $R_0$, $R_1$, $\{ a_n \}$, and $\{ b_n \}$) such that, for any complex numbers $\{ \tilde \la_j \}_{j = 1}^{2l-1}$ and for any polynomials
$$
\tilde R_j(\la) = \sum_{n = 0}^{l-j} \tilde r_{n,j}\la^n, \quad j = 0, 1,
$$
satisfying the estimate
$$
\de := \max_{j = \overline{1, 2l-1}} |\la_j - \tilde \la_j| + \max_{\substack{n = \overline{0, l-j}, \\ j = 0, 1}} |r_{n,j} - \tilde r_{n,j}| \le \eps,
$$
there exist unique complex numbers $\{ \tilde a_n \}_{n = 1}^l$ and $\{ \tilde b_n \}_{n = 1}^l$, $\tilde a_l = 1$, such that $\{ \tilde \la_j \}_{j = 1}^{2l-1}$ are the eigenvalues of the problem~\eqref{eqy}-\eqref{bcR} with the coefficients $\{ a_n \}$, $\{ b_n \}$, $R_0$, $R_1$
replaced by
$\{ \tilde a_n \}$, $\{ \tilde b_n \}$, $\tilde R_0$, $\tilde R_1$, respectively. Moreover, the following estimates hold:
$$
|a_n - \tilde a_n| \le C \de, \quad |b_n - \tilde b_n| \le C \de, \quad n = \overline{1, l},
$$
where the constant $C$ depends only on $R_0$, $R_1$, $\{ a_n \}$, and $\{ b_n \}$.
\end{thm}

\begin{thm} \label{thm:locpsi}
Let $\{ \al_n \}_{n = 1}^l$, $\{ \beta_n \}_{n = 1}^l$, $\{ \al_n^{\bullet} \}_{n = 1}^l$, $\{ \beta_n^{\bullet} \}_{n = 1}^l$ be fixed complex numbers, $\al_n \ne 0$, $\al_n^{\bullet} \ne 0$, $\al_l \ne \al_l^{\bullet}$, and let $\{ \la_j \}_{j = 1}^{2l-1}$ be the eigenvalues of the corresponding problem~\eqref{eqpsi}-\eqref{bcpsi}. Then, there exists $\eps > 0$ (depending on $\{ \al_n \}$, $\{ \beta_n \}$, $\{ \al_n^{\bullet} \}$, $\{ \be_n^{\bullet} \}$) such that, for any complex numbers $\{ \tilde \la_j \}_{j = 1}^{2l-1}$, $\{ \tilde \al_n^{\bullet} \}_{n = 1}^l$, $\{ \tilde \beta_n^{\bullet} \}_{n = 1}^l$, $\tilde \al_l$ satisfying the estimate
$$
\de := \max_{j = \overline{1, 2l-1}} |\la_j - \tilde \la_j| + \max_{n = \overline{1, l}} (|\al_n^{\bullet} - \tilde \al_n^{\bullet}| + |\beta_n^{\bullet} - \tilde \beta_n^{\bullet}|) + |\al_l - \tilde \al_l| \le \eps,
$$
there exist unique complex numbers $\{ \tilde \al_n \}_{n = 1}^{l-1}$ and $\{ \tilde \beta_n \}_{n = 1}^l$ such that $\{ \tilde \la_j \}_{j = 1}^{2l-1}$ are the eigenvalues of the problem \eqref{eqpsi}-\eqref{bcpsi} with the coefficients $\{ \al_n \}$, $\{ \be_n \}$, $\{ \al_n^{\bullet} \}$, $\{ \be_n^{\bullet} \}$ replaced by $\{ \tilde \al_n \}$, $\{ \tilde \be_n \}$, $\{ \tilde \al_n^{\bullet} \}$, $\{ \tilde \be_n^{\bullet} \}$, respectively. Moreover, the following estimates hold:
$$
|\al_n - \tilde \al_n| \le C \de, \quad |\be_n - \tilde \be_n| \le C \de, \quad n = \overline{1, l},
$$
where the constant $C$ depends only on $\{ \al_n \}$, $\{ \be_n \}$, $\{ \al_n^{\bullet} \}$, and $\{ \be_n^{\bullet} \}$.
\end{thm}

Note that Theorems~\ref{thm:locR}-\ref{thm:locpsi} do not contain implicit conditions mentioned in Remark~\ref{rem:imp}.

\begin{remark}
Our approach can be generalized to the case when \eqref{eqtpsi} contains $\tilde l$ equations, $\tilde l > l$, and $\deg(R_j) > l-j$, $j = 0, 1$. In this case, a part of the spectrum is sufficient for the recovery of the unknown coefficients. In the case $\tilde l < l$, on the contrary, one has to use additional spectral data together with the eigenvalues for solving the inverse problem.
\end{remark}

\section{Connection with Hochstadts' results} \label{sec:Hochstadt}

In this section, we discuss the relation of our results with the classical results of Hochstadt~\cite{Hoch74, Hoch79}.

In \cite{Hoch74}, the reconstruction of the tridiagonal symmetric Jacobi matrix has been studied, by using its spectrum and the spectrum of the truncated matrix obtained by deleting the first row and the first column. In our notations, this problem can be formulated as follows. Let $\{ \mu_j \}_{j = 1}^l$ and $\{ \nu_j \}_{j = 1}^{l-1}$ be the eigenvalues (counting with their multiplicities) of the boundary value problems for the system
\begin{equation} \label{eqAB}
A_{n-1} u_{n-1} + B_n u_n + A_n u_{n + 1} = \la u_n, \quad n = \overline{1, l},
\end{equation}
with the boundary conditions $u_0 = u_{l + 1} = 0$ and $u_1 = u_{l + 1} = 0$, respectively. For the second boundary value problem, we suppose that $n = \overline{2, l}$ in \eqref{eqAB}. Suppose that $A_n \ne 0$, $n = \overline{0, l}$. The coefficients $A_0$ and $A_l$ are unimportant, so we assume for definiteness that $A_0 = A_l = 1$.

\begin{ip} \label{ip:Hoch1}
Given $\{ \mu_j \}_{j = 1}^l$ and $\{ \nu_j \}_{j = 1}^{l-1}$, find $\{ A_n \}_{n = 1}^{l-1}$ and $\{ B_n \}_{n = 1}^l$.
\end{ip}

This inverse problem, studied in \cite{Hoch74}, is closely related with Inverse Problems~\ref{ip:Weyl}-\ref{ip:2sp}. Indeed, the system~\eqref{eqAB} is reduced to the form~\eqref{eqy} by the change of variables:
\begin{equation} \label{chAB}
    u_n = d_n y_n,\: n = \overline{0, l+1},  \quad a_n = A_n^2, \quad b_n = B_n, \quad n = \overline{1, l},
\end{equation}
where the coefficients $\{ d_n \}_{n = 0}^{l + 1}$ can be chosen uniquely up to a multiplicative constant so that $d_{n + 1} = A_n d_n$, $n = \overline{0, l}$.
Consequently, the spectra $\{ \mu_j \}_{j = 1}^l$ and $\{ \nu_j \}_{j = 1}^{l-1}$ defined in this section coincide with the ones defined in Section~\ref{sec:aux}. Hochstadt \cite{Hoch74} considered the case of real $A_n > 0$. In this case, the correspondence \eqref{chAB} between the coefficients $\{ a_n\}_{n = 1}^{l-1}$, $\{b_n\}_{n = 1}^l$ and $\{A_n\}_{n=1}^{l-1}$, $\{B_n\}_{n = 1}^l$ is one-to-one, so Inverse Problems~\ref{ip:Weyl}, \ref{ip:2sp}, and~\ref{ip:Hoch1} are equivalent. However, in the general case of complex coefficients $A_n \ne 0$, they are determined by $a_n$ uniquely up to the sign.

In \cite{Hoch79}, the Jacobi matrix has been constructed by using the known spectrum and a half of the matrix coefficients. For definiteness, consider the odd $l = 2m-1$. Then the inverse problem from~\cite{Hoch79} can be formulated as follows.

\begin{ip} \label{ip:Hoch2}
Given $\{ A_n \}_{n = 1}^{m-1}$, $\{ B_n \}_{n = 1}^{m-1}$, and $\{ \mu_j \}_{j = 1}^l$, find $\{ A_n \}_{n = m}^{l-1}$ and $\{ B_n \}_{n = 1}^l$.
\end{ip}

Let us show that Inverse Problem~\ref{ip:Hoch2} can be reduced to Inverse Problem~\ref{ip:R}. By using $\{ A_n \}_{n = 1}^{m-1}$ and $\{ B_n \}_{n = 1}^{m-1}$, one can find $\{ a_n \}_{n = 1}^{m-1}$ and $\{ b_n \}_{n = 1}^{m-1}$ via~\eqref{chAB}, and then construct the polynomials $[P_n(\la)]_{n = 0}^m$ defined in Section~\ref{sec:aux}. It can be easily shown that the eigenvalue problem for~\eqref{eqAB} with the boundary conditions $u_0 = u_{l + 1} = 0$ is equivalent to the following eigenvalue problem:
\begin{align*}
    & a_n y_{n + 1} + b_n y_n + y_{n-1} = \la y_n, \quad n = \overline{m, l}, \\
    & P_{m-1}(\la) y_m - P_m(\la) y_{m-1} = 0, \quad y_{l + 1} = 0,
\end{align*}
where $P_{m-1}(\la)$ and $P_m(\la)$ are relatively prime polynomials of degrees $(m-2)$ and $(m-1)$, respectively. The latter problem is similar to~\eqref{eqy}-\eqref{bcR}, so the results of Section~\ref{sec:main} (Algorithm~\ref{alg:R}, Theorems~\ref{thm:nscR} and~\ref{thm:locR}) can be applied to this problem. After recovering $\{ a_n \}_{n = m}^{l-1}$ and $\{ b_n \}_{n = m}^l$ from the spectrum $\{ \mu_j \}_{j = 1}^l$, one can uniquely find $\{ A_n \}_{n = m}^{l-1}$ and $\{ B_n \}_{n = m}^l$ from~\eqref{chAB} if $A_n > 0$. Thus, our approach based on Inverse Problem~\ref{ip:R} with polynomials in the boundary condition can be applied to the Hochstadt inverse problem by mixed data.

\medskip

{\bf Acknowledgments.} The work of the author N.P.~Bondarenko was supported by Grant 20-31-70005 of the Russian Foundation for Basic Research. The work of the author V.A. Yurko was supported by Grant 19-01-00102 of the Russian Foundation for Basic Research.

\medskip

{\bf Conflict of interest.} The authors declare that this paper has no conflict of interest.

\medskip

{\bf Authors' contributions.} Sections~\ref{sec:intr}, \ref{sec:main}-\ref{sec:Hochstadt} belong to N.P.~Bondarenko, Section~\ref{sec:aux} belongs to V.A.~Yurko.

\noindent Natalia Pavlovna Bondarenko \\
1. Department of Applied Mathematics and Physics, Samara National Research University, \\
Moskovskoye Shosse 34, Samara 443086, Russia, \\
2. Department of Mechanics and Mathematics, Saratov State University, \\
Astrakhanskaya 83, Saratov 410012, Russia, \\
e-mail: {\it bondarenkonp@info.sgu.ru}

\medskip

\noindent Vjacheslav Anatoljevich Yurko \\
Department of Mechanics and Mathematics, Saratov State University, \\
Astrakhanskaya 83, Saratov 410012, Russia, \\
e-mail: {\it yurkova@info.sgu.ru}

\end{document}